\documentclass[reqno,12pt]{amsart}
\usepackage{geometry}
\usepackage{amsmath,amssymb,mathrsfs,color}
\usepackage[colorlinks,
linkcolor=red,
anchorcolor=green,
citecolor=blue,
]{hyperref}
\usepackage[upint]{stix}
\usepackage{subfigure}
\usepackage{float}
\usepackage{times}
\usepackage{tikz}
\usetikzlibrary{intersections}
\usepackage{paralist}
\usepackage{comment}

\makeatletter
\def\setliststart#1{\setcounter{\@listctr}{#1}%
  \addtocounter{\@listctr}{-1}}
\makeatother

\makeatletter
\@addtoreset{figure}{section}
\makeatother

\usepackage{calc}
 \newtheorem{The}{Theorem}[section]
 \newtheorem{Cor}[The]{Corollary}
 \newtheorem{Lem}[The]{Lemma}
 \newtheorem{Pro}[The]{Proposition}
 \theoremstyle{definition}
 \newtheorem{defn}[The]{Definition}
 \newtheorem{Rem}[The]{Remark}
 
 \numberwithin{equation}{section}

\newcommand{\T}{\mathbb{T}}
\newcommand{\R}{\mathbb{R}}

\newcommand{\N}{\mathbb{N}}

\newcommand{\SING}{\mbox{\rm Sing}\,(u)}
\newcommand{\CUT}{\mbox{\rm Cut}\,(u)}

\title{Singularities of solutions of Hamilton-Jacobi equations}
\author{Piermarco Cannarsa \and Wei Cheng}
\address{Dipartimento di Matematica, Universit\`a di Roma ``Tor Vergata'', Via della Ricerca Scientifica 1, 00133 Roma, Italy}
\email{cannarsa@mat.uniroma2.it}
\address{Department of Mathematics, Nanjing University, Nanjing 210093, China}
\email{chengwei@nju.edu.cn}
\date{\today}
\begin{document}
\maketitle

\begin{abstract}
This is a survey paper on the quantitative analysis of the propagation of singularities for the viscosity solutions to Hamilton-Jacobi equations in the past decades. We also review further applications of the theory to various fields such as  Riemannian geometry, Hamiltonian dynamical systems and partial differential equations.
\end{abstract}

\section{Introduction}

This is a survey paper concerning the progress made for the singularities of the solutions to Hamilton-Jacobi equations in the past decades. We begin with a quote from the paper \cite{Khanin_Sobolevski2016} by Khanin and Sobolevski: 
\medskip
\begin{center}
\begin{minipage}[c]{0.9\linewidth}
	\textit{
	The evolutionary Hamilton-Jacobi equation
	\begin{equation}\label{eq:intro_KS}\tag{HJ}
	    \frac{\partial u}{\partial t}+H(t,x,\nabla u)=0
	\end{equation}
	appears in diverse mathematical models ranging from analytical mechanics to combinatorics, condensed matter, turbulence, and cosmology $\cdots$.  In many of these applications the objects of interest are described by singularities of solutions, which inevitably appear for generic initial data after a finite time due to the nonlinearity of \eqref{eq:intro_KS}. Therefore one of the central issues both for theory and applications is to understand the behavior of the system after singularities form.
	}
\end{minipage}
\end{center}
\medskip

The notion of viscosity solutions, introduced in the seminal papers \cite{Crandall_Lions1983,Crandall_Evans_Lions1984}, provides the right class of generalized solutions to study existence, uniqueness, and stability issues for problem \eqref{eq:intro_KS}. An overview of the main features of this theory can be found in the monographs \cite{Bardi_Capuzzo-Dolcetta1997} for first order equations and \cite{Fleming_Soner_book2006} for second order equations. 

It is well known that Hamilton-Jacobi equations have no global smooth solutions in general, because solutions may develop singularities due to crossing or focusing of characteristics. The persistence of singularities, i.e, once a singularity is created, it will propagate forward in time up to $+\infty$, affords an evidence of irreversibility for equation \eqref{eq:intro_KS}, while the compactness after the evolution of the associated Lax-Oleinik semi-group gives another one (\cite{Ancona_Cannarsa_Nguyen2016_1,Ancona_Cannarsa_Nguyen2016_2}).

The expected maximal regularity for solutions of \eqref{eq:intro_KS} is the local semiconcavity of $u(t,\cdot)$ for $t>0$. Indeed, semiconcave functions were used to study well-posedness for \eqref{eq:intro_KS} before the theory of viscosity solution was developed (\cite{Douglis1961,Kruzkov1975,Krylov_book1987}). Nowadays, the notion of semiconcavity has been widely used in many mathematical fields, such as \cite{Hrustalev1978,Cannarsa_Frankowska1991,Fleming_McEneaney2000,Rifford2000,Rifford2002} in control theory and sensitivity analysis, \cite{Rockafellar1982,Colombo_Marigonda2006} in nonsmooth and variational analysis, \cite{Petrunin2007} in metric geometry. Good references on semiconcave functions include the monographs \cite{Cannarsa_Sinestrari_book,Villani_book2009}. 

To our knowledge, the first paper dealing with the singularities of viscosity solutions of \eqref{eq:intro_KS} is the paper by the first author and Soner (\cite{Cannarsa_Soner1987}). Thanks to the discovery of semiconcave functions in the study of viscosity solutions of \eqref{eq:intro_KS} (\cite{Cannarsa_Soner1989}), some propagation results for general semiconcave functions were obtained in \cite{Ambrosio_Cannarsa_Soner1993}. The propagation of singularities of semiconcave functions along Lipschitz arcs was firstly studied in \cite{Albano_Cannarsa1999} and then extend to solutions of Hamilton-Jacobi equations (\cite{Albano_Cannarsa2002}). 

It is the first time in \cite{Albano_Cannarsa2002} the authors introduced the important notion of \emph{generalized characteristics} for Hamilton-Jacobi equation \eqref{eq:intro_KS}, which is a keystone for the further progress later. In one-dimensional case, the idea of generalized characteristics also comes from earlier work by Dafermos \cite{Dafermos1977} on Burgers equation. The readers can also refer to Arnold's book \cite{Arnold_book1990} on the \emph{shock wave singularities and perestroikas of Maxwell sets} and the references therein.

%\begin{defn}\label{defn:gc}
A Lipschitz curve $\mathbf{x}:[0,T]\to\Omega$, $\mathbf{x}(0)=x_0\in\SING$, is called a \textit{generalized characteristic} with respect to $(H,u)$ from $x_0$ if the following differential inclusion is satisfied
\begin{equation}\label{eq:intro_gc}
	\dot{\mathbf{x}}(t)\in\text{co}\,H_p(t,\mathbf{x}(t),D^+u(\mathbf{x}(t))),\quad a.e.,\ t\in[0,T].
\end{equation}
%\end{defn}
It was proved in \cite{Albano_Cannarsa2002} that there exists a generalized characteristic from any initial singular point $x_0$ propagating the singularities if $0\not\in\text{co}\,H_p(x_0,D^+u(x_0))$. Using the approximation method introduced by Yu (\cite{Yu2006}), the first author and Yu showed the existence of \emph{singular characteristics} (see Definition \ref{defn:ssc}) which has more regularity information. More precisely, any such a singular characteristic $\mathbf{x}$ satisfies the condition $\lim_{t\to0^+}\operatorname*{ess\ sup}_{s\in[0,t]}|\dot{\mathbf{x}}(s)-\dot{\mathbf{x}}^+(0)|=0$. 

From late 1990's, Fathi established weak KAM theory mainly for the stationary Hamilton-Jacobi equation (\cite{Fathi1997_1,Fathi1997_2,Fathi1998_1,Fathi1998_2,Fathi_book,Fathi2020})
\begin{equation}\label{eq:intro_HJs}
	H(x,Du(x))=0,\quad x\in M,
\end{equation}
where $M$ is a smooth manifold, $H$ is a Tonelli Hamiltonian and $0$ is the Ma\~n\'e's critical value. Weak KAM theory bridges Mather theory (\cite{Mather1991,Mather1993,Mane1992}) from Hamiltonian dynamical systems to the theory of viscosity solutions of \eqref{eq:intro_HJs}. Any weak KAM solution $u$ of \eqref{eq:intro_HJs} is exact the common fixed point of the associated negative type Lax-Oleinik semi-group $\{T_t\}_{t>0}$, $T_tu=u$ for all $t>0$.

For any (Lipschitz and semiconcave) weak KAM solution $u$ of \eqref{eq:intro_HJs}, an intrinsic method was developed in the paper \cite{Cannarsa_Cheng3}. Using the positive Lax-Oleinik semi-group $\{\breve{T}_t\}_{t>0}$, one can obtain an \emph{intrinsic singular characteristic} propagating singularities from any singular initial point, or general cut point of $u$ (\cite{Cannarsa_Cheng_Fathi2017}).

Although singular characteristics satisfy \eqref{eq:intro_gc}, the convex hull in the differential inclusion \eqref{eq:intro_gc} is an obvious obstacle to establish uniqueness and stability. The only well-understood system with such well-posedness properties is the system with Hamiltonians quadratic in the momentum variable. A typical example is the Hamiltonian $H=\frac 12|p|^2$, where differential inclusion \eqref{eq:intro_gc} becomes the \emph{generalized gradient} system
\begin{align*}
	\dot{\mathbf{x}}(t)\in D^+u(\mathbf{x}(t)),\quad t\in[0,T].
\end{align*}

Inspired by earlier works \cite{Bogaevsky2002,Bogaevski2006,Stromberg2013}, Khanin and Sobolevski essentially proved the existence of singular characteristics satisfying \eqref{eq:intro_gc} without convex hull under some extra conditions on the initial data (\cite{Khanin_Sobolevski2016}). These kinds of singular characteristics are called \emph{strict singular characteristics} (or \emph{broken characteristics} in \cite{Stromberg2013}). The fact that singular characteristics satisfy more restrictive dynamics than \eqref{eq:intro_gc} might help to obtain some kind of uniqueness result. Indeed, in the recent work \cite{Cannarsa_Cheng2020}, we solved such a well-posedness problem in $\R^2$ for non-critical initial data.

When we pursue applications of this theory, global propagation results for  solutions of Hamilton-Jacobi equations turn out to be necessary. Global propagation for the closure of the singular set was obtained by Albano in \cite{Albano2016_1}. For the propagation of genuine singularities, a global result for a Cauchy-Dirichlet problem with quadratic Hamiltonian was obtained in \cite{Cannarsa_Mazzola_Sinestrari2015} using an energy method. More global results for weak KAM solutions and Dirichlet problem using intrinsic method can be found in \cite{Cannarsa_Cheng3,Cannarsa_Cheng_Fathi2017,CCMW2019,Cannarsa_Cheng_Fathi2019}.

One important application of the global propagation result is the homotopy equivalence between the complement of Aubry set and the singular set of any weak KAM solution, and the local contractibility of the singular set (\cite{Cannarsa_Cheng_Fathi2017,Cannarsa_Cheng_Fathi2019}). An earlier result for such homotopy equivalence for the distance function on Riemannian manifolds was obtained in \cite{ACNS2013} based on invariance properties of the generalized gradients flow. Moreover, global propagation result in \cite{Cannarsa_Cheng_Fathi2019} can also be applied to some basic problem in Riemmanian geometry such as the analysis of the set of points which can be joined by at least two minimizing geodesics. There are also some applications of this theory to Hamiltonian dynamical systems, mainly in the frame of Mather theory and weak KAM theory (\cite{Cannarsa_Cheng_Zhang2014,Cannarsa_Cheng2,Cannarsa_Chen_Cheng2019,Zhang2020}). Above evidence suggests that the story of singularities will continue and further applications to various topics will appear in the near future.

The paper is organized as follows. In section 2, we introduce some necessary materials on Hamilton-Jacobi equations and semiconcavity. In section 3 and 4, we will review the progress in local and global propagation of singularities for various kinds of problem. In section 5, we will concentrate on the setting of the weak KAM theory, especially the applications to the topological and dynamical applications. There is also a short concluding remark in section 6. We also provide a new proof of the Lipschitz regularity for the intrinsic singular characteristics in the appendix, which looks quite natural and intuitive comparing to the original one in \cite{Cannarsa_Cheng3}.

\medskip

\noindent\textbf{Acknowledgements.} Piermarco Cannarsa was supported in part by the National Group for Mathematical Analysis, Probability and Applications (GNAMPA) of the Italian Istituto Nazionale di Alta Matematica ``Francesco Severi'' and by Excellence Department Project awarded to the Department of Mathematics, University of Rome Tor Vergata, CUP E83C18000100006. Wei Cheng is partly supported by National Natural Science Foundation of China (Grant No. 11871267, 11631006 and 11790272). The second author also thanks to Jiahui Hong for helpful discussion on some part of the appendix.

\section{Preliminaries}

%\subsection{Semiconcave functions}

Let $\Omega\subset\R^n$ be a convex open set. We recall that  a function $u:\Omega\to\R$ is {\em semiconcave} (with linear modulus) if there exists a constant $C>0$ such that
\begin{equation}\label{eq:SCC}
\lambda u(x)+(1-\lambda)u(y)-u(\lambda x+(1-\lambda)y)\leqslant\frac C2\lambda(1-\lambda)|x-y|^2
\end{equation}
for any $x,y\in\Omega$ and $\lambda\in[0,1]$. 

For any continuous function $u:\Omega\subset\R^n\to\R$ and any $x\in\Omega$, the closed convex sets
\begin{align*}
D^-u(x)&=\left\{p\in\R^n:\liminf_{y\to x}\frac{u(y)-u(x)-\langle p,y-x\rangle}{|y-x|}\geqslant 0\right\},\\
D^+u(x)&=\left\{p\in\R^n:\limsup_{y\to x}\frac{u(y)-u(x)-\langle p,y-x\rangle}{|y-x|}\leqslant 0\right\}.
\end{align*}
are called the {\em subdifferential} and {\em superdifferential} of $u$ at $x$, respectively.

The following statement characterizes semiconcavity (with linear modulus) for a continuous function by using superdifferentials. 

\begin{Pro}
\label{criterion-Du_semiconcave}
Let $u:\Omega\to\R$ be a continuous function. If there exists a constant $C>0$ such that, for any $x\in\Omega$, there exists $p\in\R^n$ such that
\begin{equation}\label{criterion_for_lin_semiconcave}
u(y)\leqslant u(x)+\langle p,y-x\rangle+\frac C2|y-x|^2,\quad \forall y\in\Omega,
\end{equation}
then $u$ is semiconcave with constant $C$ and $p\in D^+u(x)$.
Conversely,
if $u$ is semiconcave  in $\Omega$ with constant $C$, then \eqref{criterion_for_lin_semiconcave} holds for any $x\in\Omega$ and $p\in D^+u(x)$.
\end{Pro}

Let $u:\Omega\to\R$ be locally Lipschitz. We recall that a vector $p\in\R^n$ is called a {\em reachable} (or {\em limiting}) {\em gradient}  of $u$ at $x$ if there exists a sequence $\{x_n\}\subset\Omega\setminus\{x\}$ such that $u$ is differentiable at $x_k$ for each $k\in\N$, and
$$
\lim_{k\to\infty}x_k=x\quad\text{and}\quad \lim_{k\to\infty}Du(x_k)=p.
$$
The set of all reachable gradients of $u$ at $x$ is denoted by $D^{\ast}u(x)$.

\begin{Pro}\label{basic_facts_of_superdifferential}
Let $u:\Omega\subset\R^n\to\R$ be a continuous semiconcave function and let $x\in\Omega$. Then the following properties hold.
\begin{enumerate}[\rm {(}a{)}]
  \item $D^+u(x)$ is a nonempty closed convex set in $\R^n$ and $D^{\ast}u(x)\subset\partial D^+u(x)$, where  $\partial D^+u(x)$ denotes the topological boundary of $D^+u(x)$.
  \item The set-valued function $x\rightrightarrows D^+u(x)$ is upper semicontinuous.
  \item $D^+u(x)$ is a singleton if and only if $D^-u(x)\not=\varnothing$. If $D^+u(x)$ is a singleton, then $u$ is differentiable at $x$. Moreover, if $D^+u(x)$ is a singleton for every point in $\Omega$, then $u\in C^1(\Omega)$.
  \item $D^+u(x)=\mathrm{co}\, D^{\ast}u(x)$.
  \item If $u$ is both semiconcave and semiconvex in $\Omega$, then $u\in C^{1,1}(\Omega)$.
%  \item $D^{\ast}u(x)=\big\{\lim_{i\to\infty}p_i: p_i\in D^+u(x_i),\; x_i\to x,\;\mathrm{diam}\,(D^+u(x_i))\to 0\big\}$.
\end{enumerate}
\end{Pro}

\begin{defn}
Let $u:\Omega\to\R$ be a semiconcave function. $x\in\Omega$ is called a \emph{singular point} of $u$ if $D^+u(x)$ is not a singleton. The set of all singular points of $u$ is denoted by $\SING$.
\end{defn}

To study the rectifiability of the singular set $\SING$ of a semiconcave function $u$, we need some concepts from geometric measure theory.

\begin{defn}
Let $k\in\{0,1,\cdots,n\}$ and let $C\subset\R^n$. 
\begin{enumerate}[(1)]
	\item $C$ is called a \emph{$k$-rectifiable} set if there exists a Lipschitz continuous function $f:\R^k\to\R^n$ such that $C\subset f(\R^k)$.
	\item $C$ is called a \emph{countably $k$-rectifiable} set if it is the union of a countable family of $k$-rectifiable sets.
	\item $C$ is called a \emph{countably $\mathcal{H}^k$-rectifiable} set if there exists a countably $k$-rectifiable set $E\subset\R^n$ such that $\mathcal{H}^k(C\setminus E)=0$, where $\mathcal{H}^k$ stands for $k$-dimensional Hausdorff (outer) measure.
\end{enumerate}
\end{defn}

%\subsection{Hamilton-Jacobi equation, viscosity solution and weak KAM theory}

Let $\Omega\subset\R^n$ be an open set and let $H$ be a continuous real-valued function on $\Omega\times\R\times\R^n$. Let us again consider the general nonlinear first order equation
\begin{equation}\label{eq:HJ_general}
	H(t,x,u(x),Du(x))=0,\quad x\in\Omega
\end{equation}
in the unknown $u:\Omega\to\R$.

\begin{defn}
A continuous real-valued function $u$ on $\Omega$ is called a \emph{viscosity solution} of \eqref{eq:HJ_general} if for every $x\in\Omega$ and $\varphi\in C^1(\Omega,\R)$
\begin{enumerate}[(a)]
	\item $u-\varphi$ has a local maximum at $x$ implies $H(t,x,u(x),D\varphi(x))\leqslant0$, or $u$ is a \emph{viscosity sub-solution} of \eqref{eq:HJ_general};
	\item $u-\varphi$ has a local minimum at $x$ implies $H(t,x,u(x),D\varphi(x))\geqslant0$, or $u$ is a \emph{viscosity super-solution} of \eqref{eq:HJ_general}.
\end{enumerate}
\end{defn}

The relation between continuous viscosity solution and its semiconcavity is 

%\cite{Wolter1993}\cite{Choi_Choi_Moon1997}

\begin{Pro}
For $u:\Omega\to\R$ semiconcave and $H\in C(\Omega\times\R\times\R^n,\R)$,
\begin{enumerate}[\rm (i)]
	\item if u is a viscosity solution of $H(x,u(x),Du(x))=0$ in $\Omega$, then
	\begin{align*}
		H(x,u(x),p)=0,\quad\forall x\in\Omega,\ p\in D^*u(x);
	\end{align*}
	\item if $H(x,u,\cdot)$ is convex, then
	\begin{align*}
		H(x,u(x),Du(x))=0,\ a.e.,\quad\Longleftrightarrow\quad\mbox{$u$ is a viscosity solution of $H(x,u(x),Du(x))=0$};
	\end{align*}
\end{enumerate}
\end{Pro}

\section{Local propagation of singularities}

\subsection{Rectifiability of $\SING$ for semiconcave functions and viscosity solutions}

For a semiconcave function $u$ on an open subset $\Omega\subset\R^n$, $Du$ is a function of bounded variation (see, for instance, \cite{Evans_Gariepy_book}). The singular set $\SING$ coincides with the \emph{jump set} $S_{Du}$, considered as a function of $BV_{\rm loc}(\Omega,R^n)$ and is a countably $\mathcal{H}^{n-1}$-rectifiable set. Apart from earlier contributions for distance functions as in \cite{Erdos1945}, to our knowledge the first general results about the rectifiability of the singular sets of concave functions are due to Zaj\'i\v{c}ek \cite{Zajicek1978,Zajicek1979} and Vesel\'y \cite{Vesely1986,Vesely1987}.  Similar properties were later extended to semiconcave functions with general modulus in \cite{Alberti_Ambrosio_Cannarsa1992}.

%But, using completely different techniques, one can obtain more detailed rectifiability results than above. 
To obtain a fine description of $\SING$ for a semiconcave function on $\Omega$, it is convenient to introduce a hierarchy of subsets of $\SING$ according to the dimension of the superdifferential. The \emph{magnitude} of a point $x\in\Omega$ (with respect to u) is the integer $\kappa(x)=\dim D^+u(x)$. Given an integer $k\in\{0,\ldots,n\}$ we set
\begin{align*}
	\mbox{\rm Sing}^k\,(u)=\{x\in\Omega: \kappa(x)=k\}.
\end{align*}

\begin{Pro}[\cite{Cannarsa_Sinestrari_book}]\label{eq:rect1}
If $u$ is semiconcave in $\Omega$, then the set $\mbox{\rm Sing}^k\,(u)$ is countably $(n-k)$-rectifiable for any integer $k\in\{0,\ldots,n\}$.	In particular, $\SING$ is countably $(n-1)$-rectifiable.
\end{Pro}

Now, we turn to the analyze of rectifiability of $\overline{\SING}$, with $u$ the value function to the classical one free endpoint problem from calculus of variation, i.e.,
\begin{equation}\label{eq:CV_tx}\tag{CV$_{t,x}$}
	u(t,x)=\inf_{\xi\in\mathcal{A}_{t,x}}u_0(\xi(0))+\int^t_0L(s,\xi(s),\dot{\xi}(s))\ ds,\quad (t,x)\in(0,T)\times\R^n,
\end{equation}
where $L$ is a Tonelli Lagrangian of class $C^{k+1}$ ($k\geqslant 1$) and $u_0$ is of class $C^{k+1}$, and $\mathcal{A}_{t,x}$ is the set of all absolutely continuous curves $\xi:[0,t]\to\R^n$ such that $\xi(t)=x\in\R^n$.

We have already seen in Proposition \ref{eq:rect1} that $\SING$ is countably $(n-1)$-rectifiable. Recall that, under the assumption on $L$ and $u_0$, 
\begin{align*}
	\overline{\SING}=\SING\cup\mbox{\rm Conj}\,(u),
\end{align*}
where $\mbox{\rm Conj}\,(u)$ is the set of conjugate points of problem \eqref{eq:CV_tx} (see \cite[Page 155]{Cannarsa_Sinestrari_book} or \cite{Cannarsa_Mennucci_Sinestrari1997} for the definition). So, we only need to analyze the rectifiability of $\mbox{\rm Conj}\,(u)$. By a Sard type argument (\cite{Fleming1969}) one has $\mathcal{H}^{n+1/k}(\mbox{\rm Conj}\,(u))=0$. However, the above estimate does not imply the rectifiability of $\mbox{\rm Conj}\,(u)$ even if $u_0$ is of class $C^{\infty}$. 

\begin{Pro}[\cite{Cannarsa_Mennucci_Sinestrari1997}]
Under previous assumption,
\begin{enumerate}[\rm (a)]
	\item $\overline{\SING}=\SING\cup\mbox{\rm Conj}\,(u)$.
	\item $\mbox{\rm Conj}\,(u)$ is countably $\mathcal{H}^n$-rectifiable, and so is $\overline{\SING}$.
	\item $\mathcal{H}^{n-1+2/k}(\SING\setminus\mbox{\rm Conj}\,(u))=0$.
	\item If $L$ and $u_0$ is of class $C^{\infty}$, then $\dim_{\mathcal{H}}(\SING\setminus\mbox{\rm Conj}\,(u))\leqslant n-1$.
\end{enumerate}	
\end{Pro}

If the initial datum $u_0$ has weaker regularity than $C^2$, then $\overline{\SING}$ can fail to be countably $\mathcal{H}^n$-rectifiable, see \cite[Example 6.6.13]{Cannarsa_Sinestrari_book}. Notice that, in the mentioned example, $L$ is of class $C^{\infty}$. For the further progress along this line, see \cite{Pignotti2002,Mennucci2004}.

\subsection{Generalized characteristics}

Let $\Omega\subset\R^n$ be open and let $u:\Omega\to\R$ be a Lipschitz and semiconcavity viscosity solution of the Hamilton-Jacobi equation
\begin{equation}\label{eq:HJ_s_locl}
	H(x,u(x),Du(x))=0,\quad x\in\Omega.
\end{equation}
The notion of generalized characteristics with respect to $(H,u)$ plays a central r\^ole in the study the phenomenon that the singularities propagates along a Lipschitz curve from an initial point $x_0\in\SING$. 

\subsubsection{Propagation of singularities for general semiconcave functions}

Before dealing with viscosity solutions of \eqref{eq:HJ_s_locl}, we begin with a result concerning propagation of singularities for semiconcave functions with linear modulus. 

\begin{Pro}[\cite{Albano_Cannarsa1999}]\label{pro:AC1}
Let $u:\Omega\to\R$ be a semiconcave function. If $x_0\in\SING$, or
\begin{equation}\label{eq:Singlar_point2}
	\partial D^+u(x_0)\setminus D^*u(x_0)\not=\varnothing,
\end{equation}
then, there exists a Lipschitz singular arc $\mathbf{x}:[0,\tau]\to\Omega$ with $\mathbf{x}(0)=x_0$ such that $\dot{\mathbf{x}}^+(0)$ exists, $\dot{\mathbf{x}}^+(0)\not=0$ and
\begin{align*}
	\inf_{t\in[0,\tau]}\mbox{\rm diam}\,(D^+u(\mathbf{x}(t)))>0.
\end{align*} 
\end{Pro}

We note that condition \eqref{eq:Singlar_point2} is equivalent to the existence of two vectors $p_0\in D^+u(x_0)\setminus D^*u(x_0)$ and $q\in\R^n\setminus\{0\}$ such that
\begin{equation}\label{eq:normal_cone}
	\langle p-p_0,q\rangle\geqslant0,\quad\forall p\in D^+u(x_0).
\end{equation}
We will see later the importance of condition \eqref{eq:Singlar_point2} which was initially pointed out in \cite{Ambrosio_Cannarsa_Soner1993}. The key idea of the proof of Proposition \ref{pro:AC1} is to construct a function
\begin{align*}
	\phi_s(x)=u(x)-u(x_0)-\langle p_0-q,x-x_0\rangle-\frac 1{2s}|x-x_0|^2,\quad x\in\Omega.
\end{align*}
Being strictly concave for small $s>0$, $\phi_s$ has a unique maximizer $x_s$ in a small neighborhood of $x_0$ in $\Omega$. The curve $s\mapsto x_s$ is exactly the local singular arc constructed in Proposition \ref{pro:AC1}. It is rather surprising that a similar idea also works with the intrinsic singular characteristics, for the study of which the term $\frac 1{2s}|x-x_0|^2$ will be replaced by the fundamental solution.

\subsubsection{Generalized characteristics}

Applying the basic idea from \cite{Albano_Cannarsa1999} to the viscosity solutions of \eqref{eq:HJ_s_locl}, Albano and the first author introduced the notion of generalized characteristic in \cite{Albano_Cannarsa2002}.

Suppose $H:\overline{\Omega}\times\R\times\R^n\to\R$ is a continuous function satisfying the following conditions:
\begin{enumerate}[({A}1)]
	\item $p\mapsto H(x,u,p)$ is convex;
	\item for any $x\in\Omega$ and $u\in\R^n$ the function $H(x,u,\cdot)$ is uniformly quasi-convex, or the $0$-level set $\{p\in\R^n: H(x,u,p)=0\}$ contains no straight line;
	\item for any $x_1,x_2\in\Omega$, $u_1,u_2\in\R$, $p\in\R^n$
	\begin{align*}
		|H(x_1,u_1,p)-H(x_2,u_2,p)|\leqslant C_0(|x_1-x_2|+|u_1-u_2|)
	\end{align*}
	for some constant $C_0>0$;
	\item $H$ is differentiable with respect to $p$ and, for any $x_1,x_2\in\Omega$, $u_1,u_2\in\R$, $p_1,p_2\in\R^n$
	\begin{align*}
		|H_p(x_1,u_1,p_1)-H_p(x_2,u_2,p_2)|\leqslant C_1(|x_1-x_2|+|u_1-u_2|+|p_1-p_2|)
	\end{align*}
	for some constant $C_1>0$.
\end{enumerate}

\begin{Pro}[\cite{Albano_Cannarsa2002}]\label{pro:AC2}
Suppose $H$	satisfies (A1)-(A4). Let $u$ be a locally semiconcave solution of \eqref{eq:HJ_s_locl} and let $x_0\in\SING$ be such that
\begin{equation}\label{eq:noncritical_weak}
	0\not\in\mbox{\rm co}\, H_p(x_0,u(x_0),D^+u(x_0)).
\end{equation}
Then, there exists a Lipschitz arc $\mathbf{x}:[0,\tau]\to\Omega$ satisfying the following. 
\begin{enumerate}[\rm (1)]
	\item $\mathbf{x}$ is a generalized characteristic with respect to $(H,u)$ from $x_0$, that is,
	\begin{align}\label{intro:gc}
	\begin{cases}
		\dot{\mathbf{x}}(s)\in\mathrm{co}\,H_p\big(\mathbf{x}(s),D^+u(\mathbf{x}(s))\big)&\quad \text{a.e.}\;s\in[0,\tau],\\
		{\mathbf{x}}(0)=x_0.&
	\end{cases}
	\end{align}
	\item $\mathbf{x}$ is an injection.
	\item $\mathbf{x}(t)\in\SING$ for all $t\in[0,\tau]$.
	\item $\dot{\mathbf{x}}^+(0)$ exists and $\dot{\mathbf{x}}^+(0)=H_p(x_0,u(x_0),p_0)$ where $p_0=\arg\min_{p\in D^+u(x_0)}H(x_0,u(x_0),p)$.
\end{enumerate}
\end{Pro}

The proof of Proposition \ref{pro:AC2} uses the result in Proposition~\ref{pro:AC1} together with an Euler segment approximation method. Moreover, one can also derive the useful energy estimate
\begin{align*}
	H(\mathbf{x}(s),u(\mathbf{x}(s)),\mathbf{p}(s))\leqslant\frac 12H(x_0,u(x_0),p_0),\quad s\in[0,\tau],
\end{align*}
where $\mathbf{p}:[0,\tau]\to\R^n$ is defined by $\mathbf{p}(0)=p_0$ and
\begin{align*}
	\mathbf{x}(s)-x_0=s[\mathbf{p}(s)-p_0+H_p(x_0,u(x_0), p_0)],\quad \forall s\in(0,\tau].
\end{align*}
This kind of energy estimate can be used to deduce  global propagation results. 

\subsubsection{An approximation method and singular characteristics}

Needless to say, the proof of Proposition \ref{pro:AC1} and Proposition \ref{pro:AC2}  utilises techniques from nonsmooth analysis and control theory. A simpler method was introduced in \cite{Yu2006} and \cite{Cannarsa_Yu2009}. The following  approximation lemma, proved in \cite{Cannarsa_Yu2009}, will be frequently used in what follows.

\begin{Lem}\label{lem:approximation}
Given a semiconcave function on $\Omega$, we assume there are positive constants $L_i$, $i=0,1,2$, such that $|u(x)|\leqslant L_0$ for all $x\in\Omega$, $|Du(x)|\leqslant L_1$ for almost all $x\in\Omega$, and $u$ has semiconcavity constant $L_2$. Let $x_0\in\Omega$ and let $V$ be an open subset of $\Omega$ such that $x_0\in V\subset\overline{V}\subset\Omega$. Then for any $p\in D^+u(x_0)$ there is a sequence $\{u_m\}_{m\geqslant1}\subset C^{\infty}(\overline{V})$ such that
\begin{enumerate}[\rm (a)]
	\item $|u_m(x)|\leqslant L_0$, $|Du_m(x)|\leqslant L_1$, $D^2u_m(x)\leqslant L_2I_n$ for all $x\in V$,
	\item $\lim_{m\to\infty}u_m=u$ uniformly in $\overline{V}$ and $\lim_{m\to\infty}Du_m(x_0)=p$.
\end{enumerate}
\end{Lem}

The following result can be regarded as a refinement of Proposition \ref{pro:AC2}.  

\begin{Pro}[\cite{Cannarsa_Yu2009}]\label{pro:CY}
Suppose $u$ is semiconcave function on $\Omega$ and $H$ is a function of class $C^1$ satisfying (A1) and
\begin{enumerate}[\em (A2')]
	\item for any $x\in\Omega$, $u\in\R^n$ and $c\in\R$, the $c$-level set $\{p\in\R^n: H(x,u,p)=c\}$ contains no straight line.
\end{enumerate}	
Let $x_0\in\SING$ and $p_0=\arg\min_{p\in D^+u(x_0)}H(x_0,u(x_0),p)$. Then, there exists a Lipschitz arc $\mathbf{x}:[0,\tau]\to\Omega$ such that:
\begin{enumerate}[\rm (i)]
	\item $\mathbf{x}$ is a generalized characteristic for $(H,u)$ starting at $x_0$;
	\item $\mathbf{x}(t)\in\SING$ for all $t\in[0,\tau]$;
	\item $\dot{\mathbf{x}}^+(0)$ exists and $\dot{\mathbf{x}}^+(0)=H_p(x_0,u(x_0),p_0)$;
	\item $\lim_{t\to0^+}\operatorname*{ess\ sup}_{s\in[0,t]}|\dot{\mathbf{x}}(s)-\dot{\mathbf{x}}^+(0)|=0$.
\end{enumerate}
\end{Pro}

\begin{Rem}
For what follows we need further details related to Proposition \ref{pro:CY}.
\begin{enumerate}[--]
	\item The semiconcave function $u$ is not required to be a solution of \eqref{eq:HJ_s_locl}. But, if $\Omega$ is bounded, being Lipschitz, $u$ must be a subsolution of \eqref{eq:HJ_s_locl} with Hamiltonian $H-c$ for some $c\in\R$.
	\item Observe that, in general, a generalized characteristic may well be a constant arc. But, for solutions of \eqref{eq:HJ_s_locl},  it was proved in \cite{Albano_Cannarsa2002} that singularities propagate along genuine shocks (injective generalized characteristics) under assumption \eqref{eq:noncritical_weak}. If $u$ is a solution of \eqref{eq:HJ_s_locl}, as a corollary, one can show that the generalized characteristics in Proposition \ref{pro:CY} propagates singularities under the more natural condition 
	\begin{align*}
		0\not\in H_p(x_0,u(x_0),D^+u(x_0)).
	\end{align*}
	\item For the generalized characteristic $\mathbf{x}$, constructed in Proposition \ref{pro:CY}, the right-continuity of $\dot{\mathbf{x}}$ at $0$ is important for further applications. Later, we will call a singular generalized characteristic satisfying properties (i)-(iv) in Proposition \ref{pro:CY} a \emph{singular characteristic}.
\end{enumerate}
\end{Rem}

Owing to Lemma \ref{lem:approximation}, there is a sequence of smooth functions $\{u_m\}$ enjoying properties (a) and (b) in the lemma for $p=p_0$. It is easy to see that, for every $m\geqslant1$, the Cauchy problem
\begin{align*}
	\begin{cases}
		\dot{x}=H_p(x,u_m(x),Du_m(x)),&\\
		x(0)=x_0,
	\end{cases}
\end{align*}
has a $C^1$ solution $x_m:[0,\tau]\to\Omega$. Without loss of generality, we can assume that $x_m$ uniformly converges to $\mathbf{x}$ on $[0,\tau]$ as $m\to\infty$. A standard argument (see, for instance, \cite{Yu2006}) shows that $\mathbf{x}$ is a generalized characteristic for $(H,u)$ starting at $x_0$.

\subsection{Strict singular characteristics}

The r\^ole of the convex hull in the definition of generalized characteristic is quite mysterious. This is a big obstacle for us to  reveal more information about the propagation of singularities and related Hamiltonian dynamics. The next notion gets rid of such a convexity operator.

\begin{defn}\label{defn:ssc}
A Lipschitz curve $\mathbf{x}:[0,\tau]\to\Omega$ is called a \textit{strict singular characteristic} for $(H,u)$ starting at $x_0\in\SING$ if:
\begin{itemize}
\item[(i)]  denoting by 
$p(t)$ the {\em minimal energy selection} of $D^+u(\mathbf{x}(t))$, i.e.,
$$p(t)=\arg\min_{p\in D^+u(\mathbf{x}(t))}H(\mathbf{x}(t),u(\mathbf{x}(t)),p)\qquad(t\in[0,\tau]),$$
$\mathbf{x}$ satisfies
\begin{equation}\label{eq:ssc}
	\begin{split}
		\begin{cases}
		\dot{\mathbf{x}}(t)=H_p(\mathbf{x}(t),u(\mathbf{x}(t)),p(t)),& a.e.\ t\in[0,\tau],\\
		\mathbf{x}(0)=x_0\,;&
	\end{cases}
	\end{split}
\end{equation}
\item[(ii)] $\mathbf{x}(t)\in\SING$ for all $t\in[0,\tau]$;
	\item[(iii)] $\dot{\mathbf{x}}^+(0)$ exists and $\dot{\mathbf{x}}^+(0)=H_p(x_0,u(x_0),p(0))$;
	\item[(iv)] $\lim_{t\to0^+}\operatorname*{ess\ sup}_{s\in[0,t]}|\dot{\mathbf{x}}(s)-\dot{\mathbf{x}}^+(0)|=0$.
\end{itemize}

\end{defn}

The existence of strict singular characteristics for equation \eqref{eq:intro_KS} was proved in \cite{Khanin_Sobolevski2016} (see also the appendix of \cite{Cannarsa_Cheng2020}), where additional regularity properties of such curves were established, including the right-differentiability of $\mathbf{x}$ for every $t$. However, the intrinsic nature of the strict singular characteristics is still unclear. One of the most important issues of the theory is to establish the uniqueness of solutions to \eqref{eq:ssc}. We describe below a partial answer to such a fundamental problem, following the paper \cite{Cannarsa_Cheng2020}.

%Suppose $\Omega\subset\R^2$ being open. 
Hereafter in this section we assume $n=2$. Given a  semiconcave  solution $u$ of
\begin{equation}\label{eq:intro_HJ_local}\tag{HJ$_{\rm loc}$}
	H(x,Du(x))=0,\quad x\in \Omega,
\end{equation}
we denote by $\mbox{\rm Lip}^u_0(0,T;\Omega)$ the set of Lipschitz arcs $\mathbf{x}$ satisfying properties (ii), (iii), and (iv) of Definition~\ref{defn:ssc}  for all $t\in[0,T]$.

\begin{The}[\cite{Cannarsa_Cheng2020}]\label{thm:reparametrization}
Let $u$ be a  semiconcave  solution of \eqref{eq:intro_HJ_local} and let  $x_0\in \SING$ be such that $0\not\in H_p(x_0,D^+u(x_0))$. Let $\mathbf{x}_j\in\mbox{\rm Lip}_0^u(0,T;\Omega)$  ($j=1,2$)  be such that $\mathbf{x}_j(0)=x_0$.
% and
%  \begin{align*}
%		\dot{\mathbf{x}}_j^+(0)=H_p(x_0,p_0)\mbox{ where } p_0=\arg\min\{H(x_0,p): p\in D^+u(x_0)\}.
%	\end{align*}
Then, there exists $\sigma\in(0,T]$ such that there exists a unique  bi-Lipschitz homeomorphism $$\phi:[0,\sigma]\to[0,\phi(\sigma)]\subset [0,T]$$ satisfying  $\mathbf{x}_1(s)=\mathbf{x}_2(\phi(s))$ for all $s\in[0,\sigma]$. 
\end{The}

\begin{Cor}[\cite{Cannarsa_Cheng2020}]\label{Cor:reparamatrization}
Let $\mathbf{x}$ be a strict singular characteristic starting from $x_0$ and let $\mathbf{y}$ be any singular characteristic as in Proposition \ref{pro:AC2}. If $0\not\in H_p(x_0,D^+u(x_0))$, then there exists $\sigma>0$ and a bi-Lipschitz homeomorphism $\phi:[0,\sigma]\to[0,\phi(\sigma)]$ such that 
$$\mathbf{y}(\phi(s))=\mathbf{x}(s)\qquad\forall s\in[0,\sigma].$$
\end{Cor}

For strict singular characteristics, uniqueness holds without reparameterization.

\begin{The}[\cite{Cannarsa_Cheng2020}]\label{the:strict}
Let $u$ be a  semiconcave  solution of \eqref{eq:intro_HJ_local} and let  $x_0\in \SING$ be such that $0\not\in H_p(x_0,D^+u(x_0))$. Let $\mathbf{x}_j:[0,T]\to\Omega$ ($j=1,2$) be  strict singular characteristics with initial point $x_0$. Then there exists $\tau\in(0, T]$ such that $\mathbf{x}_1(t)=\mathbf{x}_2(t)$ for all $t\in[0,\tau]$.
\end{The}

Theorem~\ref{thm:reparametrization} and Theorem~\ref{the:strict} establish a connection between the absence of critical points and uniqueness of strict singular characteristics. In this direction, we also have the following global result.

\begin{Cor}[\cite{Cannarsa_Cheng2020}]\label{cor:strict}
Let $u$ be a  semiconcave  solution of \eqref{eq:intro_HJ_local} and let  $x_0\in \SING$. Let $\mathbf{x}_j:[0,T]\to\Omega$ ($j=1,2$) be  strict singular characteristics with initial point $x_0$ such that $0\not\in H_p(\mathbf{x}_j(t),D^+u(\mathbf{x}_j(t)))$ for all $t\in[0,T]$. Then $\mathbf{x}_1(t)=\mathbf{x}_2(t)$ for all $t\in[0,T]$.
\end{Cor}

\section{Global propagation of singularities}

In this section, we will discuss the global behavior of the propagation of singularities along generalized characteristics. 

\subsection{Propagating structure of the $C^1$ singular support}\label{sec:support}

A typical problem is the following evolutionary Hamilton-Jacobi equation
\begin{equation}\label{eq:HH_evo_1}
	\begin{cases}
		D_tu(t,x)+H(t,x,D_xu(t,x))=0,\quad (t,x)\in(0,T)\times\R^n\\
		u(0,x)=u_0(x),\quad x\in\R^n.
	\end{cases}
\end{equation}
A different, but related, problem is the study of the propagation of the closure of the singular set of $u$, i.e, the $C^1$ singular support of $u$.

\begin{defn}
Let $u$ be a viscosity solution of \eqref{eq:HH_evo_1}. We say that $(t,x)$ is not in the $C^1$ singular support of $u$, denoted by $(t,x)\not\in\text{sing supp}_{C^1}\,(u)$, if there exists a neighborhood $V\subset(0,T)\times\R^n$ of $(t_0,x_0)$, such that $u\in C^1(V)$. In other words, $\text{sing supp}_{C^1}\,(u)$ is the complement of the largest open set on which $u$ is of class $C^1$.
\end{defn}

Consider the system of generalized characteristics with respect to \eqref{eq:HH_evo_1}, that is,
\begin{equation}\label{eq:gc_2}
	\begin{cases}
		\dot{\mathbf{x}}(t)\in\text{co}\,H_p(t,\mathbf{x}(t),D^+u(\mathbf{x}(t))),&\quad t\in[t_0,T)\\
		\mathbf{x}(t_0)=x_0.&
	\end{cases}
\end{equation}

\begin{Pro}[\cite{Albano2014_1}]\label{pro:A2014}
Suppose $L(t,x,v)$ is a Tonelli Lagrangian with the associated Hamiltonian $H(t,x,p)$ and $u_0$ is continuous. Let $(t_0,x_0)\in\mbox{\rm sing supp}_{C^1}\,(u)$, then 
$$(t,\mathbf{x}(t))\in\mbox{\rm sing supp}_{C^1}\,(u)\qquad\forall t\in[t_0,T)\,,
$$ where $\mathbf{x}$ is solution of \eqref{eq:gc_2}.
\end{Pro}

The proof of Proposition \ref{pro:A2014} is based on an improvement of some classic results when $u_0$ is of class $C^2$ (see, for instance, Chapter 6 of \cite{Cannarsa_Sinestrari_book}). In fact, even if $u_0$ is just continuous, one can show that \emph{if $(t,x)\not\in\text{sing supp}_{C^1}\,(u)$, then the associated optimal curve $\xi$ ending at $x$ must satisfies the property that $(s,\xi(s))\not\in\text{sing supp}_{C^1}\,(u)$ for $s\in(0,t]$}. Now, suppose there exists $(t,\mathbf{x}(t))\not\in\text{sing supp}_{C^1}\,(u)$ for some $t\in(t_0,T)$. Then there exists a tubular neighborhood $V\subset\R^{n+1}$ of $\{(s,\xi(s)): s\in[t_0,t)\}$, where  $\xi$ is the optimal curve such that $\xi(t)=\mathbf{x}(t)$. Moreover, $V\cap \text{sing supp}_{C^1}\,(u)=\varnothing$. On the open set $V$, $\mathbf{x}$ and $\xi$ are essentially identified because both solve the same ordinary differential equation \eqref{eq:gc_2} (by the claim above) and satisfy the same endpoint condition. This leads to a contradiction and it follows that the $C^1$ singular support must propagate to $(t,\mathbf{x}(t))$ along the generalized characteristic $\mathbf{x}$.

\begin{Rem}
We should emphasize that the proof of Proposition \ref{pro:A2014} is based on an intrinsic approach, i.e., the argument just uses the analysis of the associated characteristics system. %If applying to the bicharacteristic system in Hamiltonian form with some extra observation, one can deduce a global and intrinsic result for the genuine propagation (see Section \ref{sec:intrinsic} below).
\end{Rem}

\subsection{Global propagation of genuine singularities}

It is quite natural to ask the question if the singularities of the viscosity solution $u$ of \eqref{eq:HH_evo_1} can propagation along the associated generalized characteristic $\mathbf{x}$ for all $t>0$. In general, the answer is negative (see, for instance, Example 5.6.7 in \cite{Cannarsa_Sinestrari_book}). 

\subsubsection{Concave initial data}

For any open subset $\Omega\subset\R^n$, consider the following Hamilton-Jacobi equation 
\begin{equation}\label{eq:HH_evo_2}
	\begin{cases}
		D_tu(t,x)+H(D_xu(t,x))=0,\quad (t,x)\in(0,+\infty)\times\Omega\\
		u(0,x)=u_0(x),\quad x\in\Omega.
	\end{cases}
\end{equation}
If $\Omega=\R^n$, $H$ is of class $C^2$, $\alpha^{-1}I_n\leqslant D^2H\leqslant\alpha I_n$ for some $\alpha>0$, and equation \eqref{eq:HH_evo_2} admits a \emph{concave} solution, then it was showed in \cite{Albano_Cannarsa2000} that  if $(t_0,x_0)\in\SING$ then there exists a Lipschitz arc $(t,\mathbf{x}(t))$, $t\in[t_0,+\infty)$, with $\mathbf{x}(t_0)=x_0$, such that $(t,\mathbf{x}(t))\in\SING$ for all $t\in[t_0,+\infty)$. We remark that $u$ is concave in $[0,T]\times\R^n$ if and only if $u_0$ is concave. So, this result is very special.

\subsubsection{Generalized gradients}\label{sec:gg}

Let $S\subset\R^n$ be closed and denote by $d_S$ the distance function from $S$. It is well known that $u=d_S$ satisfies the eikonal equation
\begin{equation}\label{eq:Dirichlet_dist}
	\begin{cases}
		|Du(x)|=1,\quad x\in\R^n\setminus S,\\
		u(x)=0,\quad x\in S.
	\end{cases}
\end{equation}
Now, let $S=\R^n\setminus\Omega$ where $\Omega$ is an open domain in $\R^n$. In this case, the system of generalized characteristics becomes the \emph{generalized gradient} system:
\begin{equation}\label{eq:g_grad}
	\begin{cases}
		\dot{\mathbf{x}}(t)\in D^+u(\mathbf{x}(t)),\quad t\in[0,T],\\
		\mathbf{x}(0)=x_0\in\Omega.
	\end{cases}
\end{equation}
This is also the case for the  Hamiltonian of mechanical system which has the form $H(x,p)=\frac 12\langle A(x)p,p\rangle+V(x)$, where $V$ is a smooth function and $A(x)$ is a positive definite symmetric $n\times n$ real matrix smoothly depending on $x$.

If $\Omega$ is a bounded open subset of $\R^n$, it was shown in \cite{ACNS2013} that the generalized gradient flow given by \eqref{eq:g_grad} propagates singularities for all $t>0$. This is also true for the case of Riemannian manifolds. A  significant application of this global propagation result to geometry is that the singular set of $d_{S}$ has the same homotopy type as $\Omega$. Further deep extension of this topological result to the weak KAM context will be discussed later. We will also discuss more general Dirichlet problem in Section \ref{sec:Dirichlet}.

\subsubsection{Mechanical systems}

Now, suppose $H(p)=\frac 12\langle Ap,p\rangle$ with $A$ a positive definite $n\times n$ real matrix. Consider the following Cauchy-Dirichlet problem  
\begin{equation}\label{eq:HH_evo_3}
	\begin{cases}
		D_tu(t,x)+H(D_xu(t,x))=0,\quad (t,x)\in(0,+\infty)\times\Omega=:Q\\
		u(t,x)=\phi(t,x),\quad (t,x)\in\partial Q.
	\end{cases}
\end{equation}
Moreover,  assume $\phi$ satisfies the following compatibility condition
\begin{equation}\label{eq:compatibility1}
	\phi(t,x)-\phi(s,y)\leqslant(t-s)L\Big(\frac{x-y}{t-x}\Big)
\end{equation}
for all $(t,x),(s,y)\in\partial Q$ such that $t>s\geqslant0$.

\begin{Pro}[\cite{Cannarsa_Mazzola_Sinestrari2015}]\label{pro:CMS2015}
Let $\phi:\overline{Q}\to\R$ be a Lipschitz continuous function satisfying \eqref{eq:compatibility1} and let $u$ be a viscosity solution of \eqref{eq:HH_evo_3}. Given $(t_0,x_0)\in\SING$, let $\mathbf{x}$ be the generalized characteristic determined by \eqref{eq:gc_2} such that $\mathbf{x}(t_0)=x_0$. Then, there exists $T\in(0,+\infty]$ such that $(s,\mathbf{x}(s))\in\SING$ for all $s\in[t_0,t_0+T)$ and $\lim_{s\to(t_0+T)^-}\in\partial Q$ whenever $T<+\infty$.
\end{Pro}

The proof of the above result relies on two main ideas that are converted in two technical results, respectively. The first one is a sharp semiconcavity estimate for a suitable transform of the solution $u$ in \cite{ACNS2013}. The second one is an inequality established showing that the full Hamiltonian associated with \eqref{eq:HH_evo_3}, that is,
\begin{align*}
	F(\tau,p)=\tau+H(p),
\end{align*}
decreases along a selection of the superdifferential of $u$, evaluated at any point of a suitable arc.

\begin{Rem}
We remark that if $\Omega=\R^n$, Proposition \ref{pro:CMS2015} directly leads to a global propagation result. For general case, the statement ensures that the singularities will have global propagation or hit the boundary $\partial Q$ (see also Section \ref{sec:Dirichlet}). 
\end{Rem}

\section{Weak KAM aspects of singularities}

In this section, we will discuss the problem of propagation of singularities in the frame of weak KAM theory (\cite{Fathi1997_1,Fathi1997_2,Fathi1998_1,Fathi1998_2,Fathi_Maderna2007,Fathi_book}) and Mather theory (\cite{Mather1991,Mather1993,Mane1992}).

\subsection{Weak KAM aspects of Hamilton-Jacobi equations}

Suppose $M$ is a smooth manifold without boundary and $TM$ (resp. $T^*M$) is the tangent (cotangent) bundle of $M$. Let $L:TM\to\R$ be a Tonelli Lagrangian, i.e., $L$ is of class $C^2$, and $L(x,\cdot)$ is strictly convex for all $x\in M$ and uniformly superlinear. Let $H:T^*M\to\R$ be the associated \emph{Tonelli Hamiltonian}. We consider the stationary Hamilton-Jacobi equation
\begin{equation}\label{eq:HJ_s}
	H(x,Du(x))=0,\quad x\in M
\end{equation}
and the evolutionary one
\begin{equation}\label{eq:HJ_e}
	\begin{cases}
		D_tu+H(x,D_xu)=0,&\quad (t,x)\in(0,+\infty)\times M,\\
		u(0,x)=\phi(x),&\quad x\in M.
	\end{cases}
\end{equation}
In equation \eqref{eq:HJ_s} we always suppose $0$ on the right side equals  Ma\~n\'e's critical value. 

The solution of equation \eqref{eq:HJ_e} can be regarded as the value function of some basic problem in the calculus of variation or optimal control. For any $x,y\in M$ and $t>0$, we denote by $\Gamma^t_{x,y}$ the set of all absolutely continuous curves $\xi:[0,t]\to M$ such that $\xi(0)=x$ and $\xi(t)=y$. We define the \emph{fundamental solution} of  \eqref{eq:HJ_e} as
\begin{align}\label{eq:action}
	A_t(x,y)=\inf_{\xi\in\Gamma^t_{x,y}}\int^t_0L(\xi,\dot{\xi})\ ds,\quad x,y\in M, t>0.
\end{align}

Recall that for any Tonelli Lagrangian, the function $(t,y)\mapsto A_t(x,y)$ is  locally semiconcave and semiconvex for small $t>0$. Moreover, the function $y\mapsto A_t(x,y)$ is convex with constant $C/t$ for small $t$ (see, for instance, \cite[Proposition B.8]{Cannarsa_Cheng3}). In symplectic geometry, $A_t(x,y)$ is also known as  \emph{generating function}. The following result is  known for  generating functions in symplectic geometry (see, for instance, \cite[Chapter 9]{McDuff_Salamon_book2017}). The readers can compare  Proposition B.8 in \cite{Cannarsa_Cheng3} (see also \cite[Theorem 4.2]{Cannarsa_Fankowska2014} for Cauchy problems) and the following lemma for fundamental solutions of Hamilton-Jacobi equations, with two analogous concepts of convexity radius and injectivity radius from Riemannian geometry.

\begin{Lem}\label{lem:C2}
For any $\lambda>0$ there exists $t_{\lambda}>0$ such that the function $(t,y)\mapsto A_t(x,y)$ is of class $C^2$ in the cone
\begin{align*}
	S_\lambda(x,t_\lambda):=\big\{(t,y)\in\R\times\R^n~:~0<t< t_\lambda,\; |y-x|<\lambda t\big\}\,.
\end{align*}
\end{Lem}

\begin{proof}
Fix $x\in\R^n$ and let $q\in\R^n$. For $t>0$ consider the Hamiltonian system
\begin{equation}\label{eq:H}%\tag{H}
	\begin{cases}
		\dot{X}(t)=H_p(X(t),P(t)),& X(0)=x,\\
		\dot{P}(t)=-H_x(X(t),P(t)),& P(0)=q.
	\end{cases}
\end{equation}
We denote the solution of \eqref{eq:H} by $(X(t,q),P(t,q))$. Define $\Phi:(0,\infty)\times\R^n\to(0,\infty)\times\R^n$ as
\begin{align*}
	\Phi(t,q)=(t,X(t,q)).
\end{align*}
Observe that $\Phi$ is of class $C^2$. The associated variational equation is
\begin{equation}\label{eq:V}%\tag{V}
	\begin{cases}
		X'_q(t,q)=H_{px}(X,P)X_q(t,q)+H_{pp}(X,P)P_q(t,q),& X_q(0,q)=0;\\
		P'_q(t,q)=-H_{xx}(X,P)X_q(t,q)-H_{xp}(X,P)P_q(t,q),& P_q(0,q)=I.
	\end{cases}
\end{equation}
Consequently, we have $X'_q(0,q)=H_{pp}(x,q)$. Since $H$ is a Tonelli Hamiltonian, we conclude that for any $R>0$ there exists $\nu(R)>0$ such that if $|q|\leqslant R$ then $H_{pp}(x,q)>\nu(R)I$. Moreover, by the Lipschitz dependence of the solution of \eqref{eq:H} and \eqref{eq:V} with respect to initial data, we obtain
\begin{equation}\label{eq:LD}
	|X'_q(t,q)-X'_q(0,q)|\leqslant C(R)t,\quad\forall t\in[0,1],\ \forall |q|\leqslant R
\end{equation}
with $C(\cdot,\cdot)>0$ nondecreasing for all variables. So
\begin{align*}
	X_q(t,q)=&\,\int^t_0X'_q(s,q)\ ds\\
	=&\,\int^t_0X'_q(0,q)\ ds+\int^t_0(X'_q(s,q)-X'_q(0,q))\ ds\\
	\geqslant&\,(\nu(R)t)I_n-\left(\frac{C(R)}2\cdot t^2\right)I_n\geqslant\bigg(\nu(R)-\frac{C(R)}2\bigg)tI_n.
\end{align*}
We conclude that for any $R>0$ there exist $\nu(R),T(R)>0$ such that
\begin{equation}\label{eq:PD}
	|q|\leqslant R\quad\Rightarrow\quad X_q(t,q)\geqslant\frac{\nu(R)}2tI_n,\quad\forall|q|\leqslant R,\ 0\leqslant t\leqslant T(R).
\end{equation}

Let $\lambda>0$. Then there exists $t_\lambda$ by Proposition B.8 in \cite{Cannarsa_Cheng3} such that $A_t(x,\cdot)$ is of class $C^{1,1}_{loc}$ in $B(x,\lambda t)$ for $0<t\leqslant t_\lambda$. For any $0<t\leqslant t_\lambda$, $y\in B(x,\lambda t)$ there exists a unique minimizer $\xi_{t,y}\in\Gamma^t_{x,y}$ for $A_t(x,y)$. Notice $|L_v(\xi_{t,y},\dot{\xi}_{t,y})|\leqslant K(\lambda)$ for some constant $K(\lambda)>0$. Let
\begin{align*}
	R_0=K(\lambda),\quad T_0=T(K(\lambda)),\quad\nu_0=\frac{\nu(K(\lambda))}{2}
\end{align*}
and fix $t_0\in(0,T_0)$. Set $q_0=L_v(x,\dot{\xi}_{t_0,y}(0))$. We want to show that
\begin{align*}
	\Phi:(0,T_0)\times\R^n\to(0,T_0)\times\R^n
\end{align*}
is locally invertible at $(t_0,q_0)$ with a $C^2$ inverse. For this we observe that
\begin{align*}
	D\Phi(t,q)=
	\begin{pmatrix}
		1&X'(t,q)\\
		0&X_q(t,q)
	\end{pmatrix}
\end{align*}
and \eqref{eq:PD} implies that
\begin{align*}
	\det D\Phi(t,q)>0,\quad (t,q)\in(0,T_0)\times B(0,R_0).
\end{align*}
Then the conclusion follows from the inverse mapping theorem.

We now claim that
\begin{equation}
	X(s,q_0)=\xi_{t_0,y}(s),\quad s\in[0,t_0].
\end{equation}
This follows from the fact that $X(\cdot,q_0)$ and $\xi_{t_0}(\cdot)$ are both solutions of the Cauchy problem
\begin{align*}
	\begin{cases}
		\frac d{ds}L_v(\xi(s),\dot{\xi}(s))=L_x(\xi(s),\dot{\xi}(s)),\quad s\in[0,t]\\
		\xi(0)=x,\ \dot{\xi}(0)=\dot{\xi}_{t_0,y}(0).
	\end{cases}
\end{align*}
Consequently, $X(t_0,q_0)=\xi_{t_0}(t_0)=y$. 
Recalling that 
$$D_tA_t(x,y)=-H(\xi_{t,y}(t),L_v(\xi_{t,y}(t),\dot{\xi}_{t,y}(t)))\quad\mbox{and}\quad D_yA_t(x,y)=L_v(\xi_{t,y}(t),\dot{\xi}_{t,y}(t))$$
one completes the proof.
\end{proof}

Whenever \eqref{eq:HJ_e} has a unique solution, such a solution satisfies the Lax-Oleinik formula. More precisely, for any $\phi:M\to\R$, any $t>0$ and  any $x\in M$ we define
\begin{equation}\label{eq:Lax-Oleinik}
	\begin{split}
		T_t\phi(x)=\inf_{y\in M}\{\phi(y)+A_t(y,x)\},\\
	\breve{T}_t\phi(x)=\sup_{y\in M}\{\phi(y)-A_t(x,y)\}.
	\end{split}
\end{equation}
Then $u(t,x)=T_t\phi(x)$ is the (unique) viscosity solution of \eqref{eq:HJ_e}. Similarly, $\breve{T}_t\phi(x)$ gives the representation of the viscosity solution of \eqref{eq:HJ_e} when replacing $H$ by $-H$. We call $\{T_t\}_{t>0}$ and $\{\breve{T}_t\}_{t>0}$ the negative and positive type Lax-Oleinik operators, respectively. Both of them are continuous semigroups on suitable function spaces of initial data.

\begin{defn}
Let $u$ be a continuous function on $M$. We say $u$ is \emph{$L$-dominated} if 
\begin{align*}
	u(\xi(b))-u(\xi(a))\leqslant \int^b_aL(\xi(s),\dot{\xi}(s))\ ds,
\end{align*}
for all absolutely continuous curves $\xi:[a,b]\to\R^n\;(a<b)$, with $\xi(a)=x$ and $\xi(b)=y$. We say such an absolutely continuous curve $\xi$ is a \emph{$(u,L)$-calibrated curve}, or a \emph{$u$-calibrated curve} for short, if the equality holds in the inequality above. A curve $\xi:(-\infty,0]\to\R^n$ is called a $u$-calibrated curve if it is $u$-calibrated  on each compact sub-interval of $(-\infty,0]$. In this case, we also say that $\xi$ is a \emph{backward calibrated curve} (with respect to $u$).
\end{defn}

\begin{Pro}[\cite{Fathi1997_1,Fathi_Maderna2007}]\label{pro:weak_KAM}
There exists a  Lipschitz  semiconcave viscosity solution of \eqref{eq:HJ_s}. Moreover, such a solution $u$ is a common fixed point of the semigroup $\{T_t\}$, i.e., $T_tu=u$ for all $t\geqslant0$.
\end{Pro}

Recall that a continuous function $u$ on $M$ is called a \emph{weak KAM solution} of \eqref{eq:HJ_s} if $T_tu=u$ for all $t>0$. The following result explains the relation between  the set of all reachable gradients and the set of all backward calibrated curves from $x$ (see, e.g., \cite{Cannarsa_Sinestrari_book} or \cite{Rifford2008} for the proof). 

\begin{Pro}\label{reachable_grad_and_backward}
Let $u:M\to\R$ be a weak KAM solution of  \eqref{eq:HJ_s} and let  $x\in M$. Then $p\in D^{\ast}u(x)$ if and only if there exists a unique $C^2$ curve $\xi:(-\infty,0]\to M$ with $\xi(0)=x$ and $p=L_v(x,\dot{\xi}(0))$, which is a backward calibrated curve with respect to $u$.
\end{Pro}

\subsection{Local propagation}

In the study of singularities of weak KAM solutions, the first issue to  address is the possible existence of isolated singular points. A typical family of Hamilton-Jacobi equations on the $n$-torus $\T^n$ is 
\begin{equation}\label{eq:cell}
	H(x,c+Du_c(x))=\alpha(c),\quad x\in\T^n,
\end{equation}
where $\alpha(c)$ is Mather's $\alpha$-function evaluated at $c\in\R^n$ (\cite{Mather1991}). For given $c\in\R^n$, $\alpha(c)$ is exactly  Ma\~n\'e's critical value for the Hamiltonian $H(x,c+p)$. Recall the function $\alpha$ is convex and superlinear. Usually, the level set $\Lambda=\arg\min_{c\in\R^n}\alpha(c)$ has a flat part. Observe that, for the one-dimensional pendulum system, there exist isolated singular points of a weak KAM solution $u_c$ if $c$ is contained in the relative interior of $\Lambda$. A criterion to ensure the non-existence of isolated singular points is $c^*\not\in\Lambda$, or
\begin{equation}
	\alpha(c^*)>\min_{c\in\R^n}\alpha(c).
\end{equation}
A confirmative result that no isolated singular point exists for a weak KAM solution $u_c$ of \eqref{eq:cell} was proved in \cite{Cannarsa_Cheng_Zhang2014} for mechanical systems using a topological argument. %The following extension to general case is new. 

\subsection{Intrinsic singular characteristics}\label{sec:intrinsic}

\subsubsection{Characteristics of weak KAM solution}

In this section, we suppose $u$ is a Lipschitz  semiconcave weak KAM solution of \eqref{eq:HJ_s} on $M=\R^n$.

In \cite{Cannarsa_Cheng3}, another kind of singular curves for $u$  is constructed as follows. First, it is shown that there exists $\lambda_0>0$ such that for any $(t,x)\in \R_+\times \R^n$ and any maximizer $y$  for the function $u(\cdot)-A_t(x,\cdot)$, we have that $|y-x|\leqslant\lambda_0 t$. Then, taking $\lambda=\lambda_0+1$, one shows  that, for some  $t_{0}>0$ and any $t\in(0,t_0]$, there exists a unique  $y_{t,x}\in B(x,\lambda t)$ such that
\begin{equation}\label{eq:sup_max_rep}
	\breve{T}_tu(x)=u(y_{t,x})-A_t(x,y_{t,x}).
\end{equation}
Moreover, $u(\cdot)-A_t(x,\cdot)$  is concave with constant $C_1-C_2/t<0$ for $0<t\leqslant t_0$. For any fixed $x\in\R^n$ define 
\begin{equation}\label{eq:curve_max}
	\mathbf{x}(t)=
	\begin{cases}
		x,&t=0,\\
		y_{t,x},& t\in(0,t_0].
	\end{cases}
\end{equation}

\begin{Pro}[\cite{Cannarsa_Cheng3}]\label{pro:y}
Let   $x\in\R^n$ and let $\mathbf{x}$ be the curve defined in \eqref{eq:curve_max}. Then, the following  holds:
\begin{enumerate}[\rm (1)]
	\item $\mathbf{x}$ is Lipschitz continuous;
	\item if $x\in\SING$, then $\mathbf{x}(t)\in\SING$ for all $t\in[0,t_0]$;
	\item $\dot{\mathbf{x}}^+(0)$ exists and
	\begin{align*}
		\dot{\mathbf{x}}^+(0)=H_p(x_0,p_0)
	\end{align*}
	where $p_0=\arg\min\{H(x_0,p): p\in D^+u(x)\}$.
\end{enumerate}  
\end{Pro}

Hereafter, we  refer to the arc $\mathbf{x}$ defined in \eqref{eq:curve_max} as the {\em intrinsic characteristic} from $x$. Notice that $t_0$ is independent of the initial point. Thus, when $x\in\SING$, Proposition \eqref{pro:y}  yields global propagation of singularities.

The reader can compare to the idea of the proof---that we outline below---to the argument used to deduce the propagation of the $C^1$ singular support  in Section \ref{sec:support}. Suppose $x\in\SING$ but $y_{t,x}\not\in\SING$ ($0<t\leqslant t_0$). Applying  Fermat's rule, we have that $D_yA_t(x,y_{t,x})=Du(y_{t,x})$. Invoking Proposition \ref{reachable_grad_and_backward} and the differentiability property of the fundamental solution for small time, we conclude that there exist two minimal curves. One is the backward calibrated curve $\gamma_{t,x}:(-\infty,t]\to\R^n$ such that $\gamma_{t,x}(t)=y_{t,x}$ and $Du(y_{t,x})=L_v(\gamma_{t,x}(t),\dot{\gamma}_{t,x}(t))$. The other one is the unique minimal curve $\xi_{t,x}\in\Gamma^t_{x,y_{t,x}}$. Thus
\begin{align*}
	L_v(\xi_{t,x}(t),\dot{\xi}_{t,x}(t))=D_yA_t(x,y_{t,x})=Du(y_{t,x})=L_v(\gamma_{t,x}(t),\dot{\gamma}_{t,x}(t)).
\end{align*}
It follows that $\gamma_{t,x}$ and $\xi_{t,x}$ coincide on $[0,t]$ since both of them are extremal curves for the action functional in \eqref{eq:action} and satisfy the same endpoint conditions at $y_{t,x}$. This leads to a contradiction since we suppose $y_{t,x}\not\in\SING$. This argument shows the stronger result that any $x$ in  the \emph{cut locus} of $u$, $\CUT$, is the initial point of a singular arc which remains singular at all times. We will emphasise this point in the next section.

Now, we want to give a new proof of the Lipschitz continuity of  intrinsic  characteristics, by a reasoning that seems more natural than the method we used in \cite{Cannarsa_Cheng3}. This proof is based on the combination of the approximation argument used in \cite{Cannarsa_Yu2009} and regularity of the fundamental solution (Lemma \ref{lem:C2}).

First, suppose $\phi\in C^2(\Omega)$ with $\Omega$ a  bounded  open subset of $\R^n$. Fix $x\in\Omega$. By following the approach in \cite{Cannarsa_Cheng3}, we have that there exists $t_0>0$ such that for any $t\in(0,t_0]$ the strictly concave function $\phi(\cdot)-A_t(x,\cdot)$ attains its maximum at a unique point $y(t)$. In other words, the curve $y$ satisfies the equation
\begin{equation}\label{eq:gc_eq}
	F(t,y(t))
	%:=D\phi(y(t))-D_yA_t(x,y(t))
	=0,\quad t\in(0,t_0),
\end{equation}
where $F(t,y):=D\phi(y)-D_yA_t(x,y)$ is of class $C^1$. 
Let $\xi_{t,y}\in\Gamma^t_{x,y}$ be the unique minimizer for the functional defining $A_t(x,y)$. Recall that  
$$D_yA_t(x,y)=L_v(\xi_{t,y}(t),\dot{\xi}_{t,y}(t))=:p(t,y)\quad\mbox{and}\quad D_tA_t(x,y)=-H(y,p(t,y)).$$ 
Then, $D_yp(t,y)=D^2_{y}A_t(x,y)$ and
\begin{align*}
	D_tF(t,y)=&\,-D_tD_yA_t(x,y)=-D_yD_tA_t(x,y))\\
	=&\,D_xH(y,p(t,y))+D_pH_p(y,p(t,y))D_yp(t,y),\\
	D_yF(t,y)=&\,D^2\phi(y)-D^2_yA_y(x,y)=D^2\phi(y)-D_yp(t,y).
\end{align*}
Thus, by differentiating \eqref{eq:gc_eq} with respect to $t$ we obtain
\begin{align*}
	\big(D^2_{y}A_t(x,y(t))
	-D^2\phi(y(t))\big)\dot{y}(t)=D_xH(y(t),p(t,y(t)))+D_pH(y(t),p(t,y(t)))D_yp(t,y(t)).
\end{align*}
Notice that $D^2_{y}A_t(x,y)-D^2\phi(y)$ is invertible since
\begin{align*}
	D^2_{y}A_t(x,y)-D^2\phi(y)>\Big(\frac{C_2}t-C_1\Big)I>\frac{C_3I}t\,,\quad 0<t\leqslant t_0.
\end{align*}
Set $B(t,y)=\big(D^2_{y}A_t(x,y)-D^2\phi(y)\big)^{-1}$. Then $B(t,y)$ is positive definite and  $B(t,y)< \frac {tI}{C_3}$. 
So,  
\begin{equation}
	\dot{y}(t)=B(t,y(t))\big(D_xH(y(t),p(t,y(t)))+D_pH(y(t),p(t,y(t)))D^2_{y}A_t(x,y(t))\big).
\end{equation}
By Lemma 3.3 in \cite{Cannarsa_Cheng3} we have that $\{\dot{\xi}_{t,y}\}_{(0,t_0]}$ is an equi-Lipschitz family. Hence, 
$$|D_xH(y(t),p(t,y(t)))+D_pH(y(t),p(t,y(t)))D^2_{y}A_t(x,y(t))|\leqslant C_4\,,$$  
where $C_4$ at most depends on $\text{Lip}\,(\phi)$.
Therefore, we conclude that
\begin{equation}\label{eq:equi_Lip}
	|\dot{y}(t)|\leqslant \frac{C_4}{C_3}\,t,\quad t\in(0,t_0].
\end{equation}

%{\color{red}
Now, suppose $u:\overline\Omega\to\R$ is a  Lipschitz  semiconcave solution of
\begin{equation}\label{eq:HJ_app}
	H(x,Du(x))=0,\quad x\in\Omega,
\end{equation}
satisfying
\begin{align*}
	|u(x)|\leqslant C_0,\quad |Du(x)|\leqslant C_1,\quad D^2u(x)\leqslant C_2I_n,\quad x\in\overline{\Omega}.
\end{align*}
%Fix $x\in\Omega$. 
Take any sequence of $C^\infty$-functions $\{u_m\}$ such that
\begin{equation}\label{eq:equi_app}
	|u_m(x)|\leqslant C_0,\quad |Du_m(x)|\leqslant C_1,\quad D^2u_m(x)\leqslant C_2I_n,\quad\forall x\in\overline{\Omega},
\end{equation}
converging uniformly to $u$ on $\overline{\Omega}$ as $m\to\infty$ (for instance, the sequence given by Lemma~\ref{lem:approximation}). As was observed above, the sequence of curves
\begin{equation}\label{eq:y_m}
	y_m(t)=
	\begin{cases}
		\arg\max\big\{u_m(y)-A_t(x,y\big)\} ,&t\in(0,t_0]\\
		x,& t=0.
	\end{cases}
\end{equation}
is well defined for some $t_0>0$.
\begin{The}
Let $u$ be a Lipschitz and semiconcave solution of \eqref{eq:HJ_app}. Let $y(t)$ be the intrinsic singular characteristic defined on $[0,t_0]$ starting from a given point $x\in\SING$, and let $y_m(t)$ be the curve defined in \eqref{eq:y_m}. Then $\{y_m\}$ converges to $y$ uniformly on $[0,t]$ and $y$ is Lipschitz continuous on $[0,t_0]$.
\end{The}

\begin{Rem}
Since the sequence of functions $u_m(\cdot)-A_t(x,\cdot)$ converges to $u(\cdot)-A_t(x,\cdot)$ uniformly as $m\to\infty$ and the family is equi-Lipschitz, then it is straightforward to see that the (unique) maximizer of $u_m(\cdot)-A_t(x,\cdot)$ converges to the maximizer of $u(\cdot)-A_t(x,\cdot)$ uniformly with respect to $t$. However, we give a detailed proof of this fact below, in order to establish a precise estimate of the convergence rate.
\end{Rem}

\begin{proof}
Let $p_t=L_v(\xi_{t,y}(t),\dot{\xi}_{t,y}(t))$, then $p_t=D_yA_t(x,y(t))\in D^+u(y(t))$. By the semiconcavity of $u$ and the convexity of the fundamental solution we deduce that
\begin{align*}
	0\leqslant&\,[u_m(y_m(t))-A_t(x,y_m(t))]-[u_m(y(t))-A_t(x,y(t))]\\
	=&\,[u_m(y_m(t))-u_m(y(t))]-[A_t(x,y_m(t))-A_t(x,y(t))]\\
	\leqslant&\,[u_m(y_m(t))-u(y_m(t))]
	+[u(y_m(t))-u(y(t))]+[u(y(t))-u_m(y(t))]\\
	&\,-\left[\langle p_t,y_m(t)-y(t)\rangle+\frac {C_3}t|y_m(t)-y(t)|^2\right]\\
	\leqslant&\,[u_m(y_m(t))-u(y_m(t))]+[u(y(t))-u_m(y(t))]+\left(C_2-\frac{C_3}t\right)|y_m(t)-y(t)|^2.
\end{align*}
So,
\begin{align*}
	\left(\frac{C_3}t-C_2\right)|y_m(t)-y(t)|^2\leqslant2\|u_m-u\|_{\infty}.
\end{align*}
Recall that $t_0$ is chosen that that $C_2-C_3/t<0$ for $t\in(0,t_0]$. Recall that the family $\{y_m\}$ is equi-Lipschitz by \eqref{eq:equi_Lip}. This implies $y_m$ converges to $y$ uniformly on $[0,t_0]$. 
\end{proof}
%}

\begin{Rem}\label{rem:Lasry-Lions}
The method used here is closely related to the Lasry-Lions regularization from convex analysis (\cite{Attouch_book,Attouch_Aze1993}) and PDE (\cite{Lasry_Lions1986}). In a weak KAM context, this method was also widely used as an interaction of the positive-negative Lax-Oleinik operators (\cite{Bernard2007,Bernard2010,Bernard2012,Fathi_Zavidovique2010}). The relation between Lasry-Lions regularization and generalized characteristics was also studied in \cite{Chen_Cheng2016,Chen_Cheng_Zhang2018}. This method was applied to minimal homoclinic orbits with respect to the Aubry set (\cite{Cannarsa_Cheng2}).
\end{Rem}

\subsubsection{Dirichlet problem}\label{sec:Dirichlet}

The proof of Proposition \ref{pro:y} actually affords a method to handle various kind of problems for propagation of singularities if the solution can be represented in the form of an inf-convolution. For example, in \cite{CCMW2019}, a global result for the Dirichlet problem was obtained using the above intrinsic approach. 
%The readers can compare to the result in Proposition \ref{pro:CMS2015}.

Consider the Dirichlet boundary-value problem for a first-order Hamilton-Jacobi equation
\begin{equation}\label{eq:HJ_d}
	\begin{cases}
		H(x,Du(x))=0,\quad x\in\Omega,\\
		u\vert_{\partial\Omega}=g.
	\end{cases}
\end{equation}
where $\Omega\subset\R^n$ is a bounded Lipschitz domain, $H$ is a Tonelli Hamiltonian, and $g$ is the boundary datum. For any $x,y\in\Omega$ and any $s<t$, we define the set of admissible arcs from $x$ to $y$ as
\begin{align*}
	\Gamma^{s,t}_{x,y}(\Omega)=\{\xi\in W^{1,1}([s,t];\R^n) : \xi(\tau)\in\overline{\Omega},\ \forall \tau\in [s,t]; \xi(s)=x; \xi(t)=y\} .
\end{align*}
For any $x,y\in\Omega$ and $t>0$, we define the fundamental solution $A^{\Omega}_t(x, y)$ relative to $\Omega$, Ma\~n\'e's potential $\Phi^{\Omega}_L(x,y)$ relative to $\Omega$, and critical value $c_{\Omega}(L)$ relative to $\Omega$  by
\begin{align*}
A^{\Omega}_t(x,y)&:=\inf_{\xi\in\Gamma^{0,t}_{x,y}(\overline{\Omega})}\int^t_0L(\xi(s),\dot{\xi}(s))\ ds,\\
\Phi^{\Omega}_L(x,y)&:=\inf_{t>0}A_t^{\Omega}(x,y),\quad c_{\Omega}(L):=-\inf_{t>0,x\in\overline{\Omega}}\frac 1tA^{\Omega}_t(x,x).
\end{align*}

Let $u$ be the value function of the following problem:
\begin{equation}\label{eq:represent_formulae2}%\tag{CV$_d$}
	u(x)=\inf_{y\in\partial\Omega}\big\{g(y)+\Phi^{\Omega}_L(y,x)\big\},\quad x\in\overline{\Omega},
\end{equation}
where $g:\partial\Omega\to\R$ is a  continuous function satisfying the compatibility condition
\begin{equation}\label{eq:supcritical1i}%\tag{B}
	g(x)-g(y)\leqslant\Phi^{\Omega}_L(y,x),\quad \forall x,y\in\partial\Omega.
\end{equation}
Observe that the function $u$ given by \eqref{eq:represent_formulae2} is the value function of an optimal exit problem  (see, for instance, \cite{Bardi_Capuzzo-Dolcetta1997}) and  a viscosity solution of \eqref{eq:HJ_d}. The following result can be regarded as an extension of Proposition~\ref{pro:CMS2015}.

\begin{Pro}[\cite{CCMW2019}]
Suppose the energy condition
\begin{align*}
	c_{\Omega}(L)<0.
\end{align*}
Let $x_0\in\CUT$. Then, the following alternative holds:
	\begin{enumerate}[\rm (a)]
	\item either there exists a generalized characteristic $\mathbf{x}:[0,+\infty)\to\Omega$ starting from $\mathbf{x}(0)=x_0$ such that $\mathbf{x}(s)\in\SING$ for all $s\in[0,+\infty)$,
	\item or there exist $T>0$ and a generalized characteristic $\mathbf{x}:[0,T)\to\Omega$ starting from $\mathbf{x}(0)=x_0$ such that $\mathbf{x}(s)\in\SING$ for all $s\in[0,T)$, and a sequence of positive real numbers $\{s_k\}$ such that
	$$
	\lim_{k\to\infty}s_k=T,\quad\text{and}\quad\lim_{k\to\infty}d_{\partial\Omega}(\mathbf{x}(s_k))=0.
	$$
	\end{enumerate}	
\end{Pro}

To exclude the case that the singularities hit the boundary we need more conditions on $\partial \Omega$. 
We shall suppose the following, where we denote $\partial \Omega$ by $\Gamma$:
\begin{enumerate}[(\bf{G}1)]
    \item  there exists $\nu\in[0,1)$ such that $g(y_1)-g(y_2)\leqslant \nu\Phi^{\Omega}_L(y_2,y_1)$, $\forall y_1,\, y_2\in\partial \Gamma$;
	%\item for any $y_1,y_2,\bar{y}\in\partial\Omega$, we have
  %$$
  %g(y_1)+g(y_2)-2g(\bar{y})-2\Phi^{\Omega}_L\left(\bar{y},\frac{y_1+y_2}2\right)\leqslant 0;
  %$$
  \item there exists $G\in C^{1,1}(\Gamma_\delta)$ for some $\delta>0$ such that $g=G|_\Gamma$ and  
  \begin{align}\label{g2}
  \langle \nabla G(x),x-y\rangle\leqslant \breve{C}|x-y|^2\qquad\forall x,y\in\Gamma
  \end{align}
  for some $\breve{C}>0$, where  $\Gamma_\delta$ denotes the $\delta$-neighborhood of $\Gamma$.
\end{enumerate}
%A typical system satisfying ({\bf{G}1}) is the one whose Lagrangian is  symmetric in the velocity variable.

\begin{Pro}
Let $\Omega\subset\R^n$ be a bounded domain with $C^2$ boundary, let $L$ be a Tonelli Lagrangian satisfying  $L\geqslant\alpha >0$ and let $g$ satisfy ({\bf G1}),({\bf G2}). If $x_0\in\CUT$, then there exists a generalized characteristic $\mathbf{x}:[0,+\infty)\to\Omega$ starting from $\mathbf{x}(0)=x_0$ such that $\mathbf{x}(s)\in\SING$ for all $s\in[0,+\infty)$.
\end{Pro}

\begin{Rem}
We note that the energy condition  $c_{\Omega}(L)<0$  (which is implicitly assumed even in the above proposition as a consequence of  the hypothesis $L\geqslant\alpha >0$) ensures that any optimal curve touches the boundary  in finite time in the associated optimal exit time problem. On the other hand, the case of $c_{\Omega}(L)=0$ is still open, especially the analysis of the  Aubry set on the boundary. For a state constrained problem,  weak KAM aspects of the boundary behaviour of solutions were studied in \cite{CCMW2020}. 
\end{Rem}

\subsection{Topology of $\SING$ and $\CUT$}

Recall the homotopy equivalence between a bounded open subset $\Omega\subset\R^n$ and the singular set of the distance function $d_{\Omega}$ discussed in Section \ref{sec:gg} is based on a global propagation result for the generalized gradient flow. It is quite natural to use the global result in the last section to study the similar problem in the weak KAM setting.

\subsubsection{Aubry set and cut locus}

Let $M$ be compact and $u$ be a weak KAM solution of \eqref{eq:HJ_s}. We define the projected Aubry set $\mathcal{I}(u)$ of $u$ as the subset of $M$ such that $x\in\mathcal{I}(u)$ if there exists a $u$-calibrated curve $\gamma:(-\infty,+\infty)\to M$ passing though $x$. We also define the cut locus of $u$, denoted by $\CUT$, as the set of points $x\in M$ where no backward $u$-calibrated curve ending at $x$ can be extended to a $u$-calibrated curve beyond $x$. In general we have the following inclusions:
\begin{align*}
	\SING\subset\CUT\subset M\setminus\mathcal{I}(u),\quad \SING\subset\CUT\subset \overline{\SING}.
\end{align*}

Using the construction in the last section, one can obtain a continuous homotopy $F:M\times[0,t]\to M$, $t>0$, with the following properties:
\begin{enumerate}[(a)]
	\item for all $x\in M$ we have $F(x,0)=x$;
	\item if $F(x,s)\not\in\SING$ for some $s>0$ and $x\in M$, then the curve $\sigma\mapsto F(x,\sigma)$ is $u$-calibrating on $[0,s]$;
	\item if there exists a $u$-calibrated curve $\gamma:[0,s]\to M$ with $\gamma(0)=x$, then $\sigma\mapsto F(x,\sigma)=\gamma(\sigma)$ for every $\sigma\in[0,\min\{s,t\}]$.
\end{enumerate}

\begin{Pro}[\cite{Cannarsa_Cheng_Fathi2017}]\label{pro:homotopy}
The inclusions 
$$\SING\subset\CUT\subset (M\setminus\mathcal{I}(u))\cap\overline{\SING}\subset M\setminus\mathcal{I}(u)$$ are all homotopy equivalences. As a consequence, for every connected component $C$ of $M\setminus\mathcal{I}(u)$ the three intersections $\SING\cap C$, $\CUT\cap C$, and $\overline{\SING}\cap C$ are path connected.
\end{Pro}

Similar to the homotopy constructed above, for any open subset $O\subset M\setminus\mathcal{I}(u)$, we can also construct a local homotopy by $G_O(x,s)=F(x,s\alpha_O(x))$, where $\alpha_O$ is continuous interpolation of the cut time function $\tau:M\to[0,+\infty]$ (the supremum of the time $t\geqslant0$ such that there exists a $u$-calibrated curve $\gamma:[0,t]\to M$, with $\gamma(0)=x$) and the local exit function $\eta_O:O\to[0,+\infty]$ defined by $\eta_O(x)=\sup\{t\in[0,+\infty): F(x,s)\in O,\ \text{for all}\ s\in[0,t]\}$. Notice $\tau<\eta_O$ on an open subset of $O$, and $\tau$ is upper semicontinuous, $\eta_O$ is lower semicontinuous.

\begin{Pro}[\cite{Cannarsa_Cheng_Fathi2017}]
The spaces $\SING$ and $\CUT$ are locally contractible.
\end{Pro}

\subsubsection{Singular set on noncompact manifolds}

Let $0<T\leqslant \infty$ and suppose $M$ is a noncompact manifold and $L$ (resp. $H$) is a Tonelli Lagrangian (resp. Hamiltonian). We will review some  topological results for the singular set of a  uniformly continuous viscosity solution of 
\begin{equation}\label{eq:HJ_e_noncompact}
	D_tu+H(x,D_xu)=0\quad\mbox{on}\quad (0,T)\times M,
\end{equation}
which were obtained in \cite{Cannarsa_Cheng_Fathi2019}, together with their applications to Riemannian geometry.

\begin{Pro}[\cite{Cannarsa_Cheng_Fathi2019}]
Let $H:T^*M\to\R$ be a Tonelli Hamiltonian. If $u$
%, defined on the open subset $O\subset\R\times M$, 
is a continuous viscosity solution of the evolutionary Hamilton-Jacobi equation \eqref{eq:HJ_e_noncompact}, then the set $\SING $ is locally contractible in $(0,T)\times M$.	
\end{Pro}

To formulate the global homotopy equivalence result, we need extend the notion of Aubry set of $u$ as follows: let $u:(0,T)\times M\to\R$, with $T\in(0,+\infty]$, be a viscosity solution of the evolutionary Hamilton-Jacobi equation \eqref{eq:HJ_e_noncompact}. The Aubry set $\mathcal{I}_T(u)$ of $u$ is the set of points $(t,x)\in (0,T)\times M$ for which we can find a curve $\gamma:[0,T)\to M$, with $\gamma(t)=x$ and
\begin{align*}
	u(b,\gamma(b))-u(a,\gamma(a))=\int^b_aL(\gamma(s),\dot{\gamma}(s))\ ds,
\end{align*}
for every $a<b\in[0,T)$.

\begin{Pro}[\cite{Cannarsa_Cheng_Fathi2019}]
Let $H:T^*M\to\R$ be a Tonelli Hamiltonian. Assume that the uniformly continuous function $u:[0,T)\times M\to\R$, with $T\in(0,+\infty]$, is a viscosity solution  of the evolutionary Hamilton-Jacobi equation \eqref{eq:HJ_e_noncompact}. Then the inclusion 
$$\mbox{\rm Sing}_T\,(u)=\SING\cap\big((0,T)\times M\big)\subset\big((0,T)\times M\big)\setminus\mathcal{I}_T(u)$$ 
is a homotopy equivalence.	
\end{Pro}

Notice that we just assume the solution $u$ of \eqref{eq:HJ_e_noncompact} to be uniformly continuous without any extra conditions on the initial data. So, there are a lot of  technical points one needs to clear in order to deal with arbitrary initial conditions (see \cite{Fathi2020}). 
%Recall that, for an arbitrary function $\phi:M\to\R$, we can define the \emph{Lax-Oleinik evolutions} $T_t\phi$ and $\breve{T}_t\phi$ ($t>0$) as \eqref{eq:Lax-Oleinik}. Since both $T_t\phi$ and $\breve{T}_t\phi$ are marginal functions, and the infimum (resp. supremum) in $T_t\phi$ (resp. $\breve{T}_t\phi$) can be achieved under very general conditions. So, one can deduces the semiconcavity/semiconvexity properties of $T_t\phi$ and $\breve{T}_t\phi$ respectively from those properties of fundamental solution $A_t(x,y)$ for short time, or from the Lasry-Lions type regularization procedure in weak KAM context.

\subsubsection{Applications to Riemmanian geometry}

Now, suppose $(M,g)$ is a complete Riemannian manifold, and $d_C$ is the distance function to a closed subset $C\subset M$. We denote by $\text{Sing}^*\,(d_C)$ the set of points in $M\setminus C$ where $d_C$ is not differentiable. 

\begin{Pro}[\cite{Cannarsa_Cheng_Fathi2019}]\label{pro:dist_sing}
If $C$ is a closed subset of a complete Riemannian manifold $(M,g)$, then $\text{Sing}^*\,(d_C)$ is locally contractible.	
\end{Pro}

In classical Riemmanian geometry, for any $x\in M$ one denotes by $\text{Cut}_{(M,g)}\,(x)$ the cut locus with respect to $x$.  It is well-known that, when $M $is compact, such a  cut locus $\text{Cut}_{(M,g)}\,(x)$ is a deformation retract of $M\setminus\{x\}$, therefore it is locally contractible. On the other hand, very little was known up to now about the set
\begin{align*}
	\mathcal{U}(M,g)=\{(x,y)\in M\times M: \ \text{there exists a unique minimal $g$-geodesic between $x$ and $y$}\}.
\end{align*}
As Marcel Berger wrote in \cite[Page 284]{Berger_book2003}
\medskip
\begin{center}
\begin{minipage}[c]{0.9\linewidth}
	\textit{The difficulty for all these studies is an unavoidable dichotomy for cut points: the mixture of points with two different segments and conjugate points}.
\end{minipage}
\end{center}
\medskip
We now proceed to explain how to distinguish the study of these two sets by using the above methods.
We will begin with another consequence of Proposition \ref{pro:dist_sing}, for which we need the following definition: for a complete Riemannian manifold $(M,g)$, we define
%\begin{align*}
%	\mathcal{U}(M,g)=\{(x,y)\in M\times M: \ \text{there exists a unique minimal $g$-geodesic between $x$ and $y$}\}.
%\end{align*}
%and
\begin{align*}
	\mathcal{NU}(M,g)=(M\times M)\setminus\mathcal{U}(M,g).
\end{align*}
The set $\mathcal{U}(M,g)$ contains a neighborhood of the diagonal $\Delta_M\subset M\times M$. In fact, we have $\mathcal{NU}(M,g)=\text{Sing}^*(d_{\Delta_M})$, the set of singularities in $(M\times M)\setminus\Delta_{M}$ of the distance function of points in $M\times M$ to the closed subset $\Delta_M$. Therefore, Proposition \ref{pro:dist_sing} implies:

\begin{Pro}[\cite{Cannarsa_Cheng_Fathi2019}]
For every complete Riemannian manifold $(M,g)$, the set 
$$\mathcal{NU}(M,g)\subset\big(M\times M\big)\setminus\Delta_{M}$$ 
is locally contractible. In particular, the set $\mathcal{NU}(M,g)$ is locally path connected.	
\end{Pro}

For a closed subset $C\subset M$, we define its Aubry set $\mathcal{A}^*(C)$ as the set of points $x\in M\setminus C$ such that there exists a curve $\gamma:[0,+\infty)\to M$ parameterized by arc-length such that $d_C(\gamma(t))=t$ and $x=\gamma(t_0)$ for some $t_0>0$.

\begin{Pro}[\cite{Cannarsa_Cheng_Fathi2019}]
If C is a closed subset of the complete Riemannian manifold $(M,g)$, then the inclusion 
$$\text{Sing}^*(C)\subset M\setminus\big(C\cup\mathcal{A}^*(C)\big)$$ 
is a homotopy equivalence.	
\end{Pro}

We remark that if $U$ is a bounded connected component of $M\setminus C$, then $U\cap\mathcal{A}^*(C)=\varnothing$, and
$$\text{Sing}^*(C)\cap U\subset U$$ 
is a homotopy equivalence (see also \cite{Lieutier2004} and Section \ref{sec:gg}). As for  unbounded components, see also \cite{Cannarsa_Peirone2001} for the Euclidean case.

\begin{Pro}[\cite{Cannarsa_Cheng_Fathi2019}]
For every compact connected Riemannian manifold $(M,g)$, the inclusion 
$$\mathcal{NU}(M,g)\subset\big(M\times M\big)\setminus\Delta_M$$ 
is a homotopy equivalence. Therefore the set $\mathcal{NU}(M,g)$ is path connected and even locally contractible.
\end{Pro}

\section{Concluding remarks}

The study of singularities of solutions to HJ equation has made remarkable progress in the past decades. Many results that seemed impossible have been obtained, and connections with other domains have been established. Nevertheless, many interesting problems remain open. Some open problems were proposed in \cite{Cannarsa_Cheng2018}.

In \cite{Cannarsa_Cheng2020}, the uniqueness of strict singular characteristic on $M=\R^2$ is proved when the initial point is not a critical point. However, the uniqueness issue is still open for higher dimensional manifolds. %Moreover, the global propagation result in Proposition \ref{a} is still not well understood because this relies on the further explanation of the essence of critical points. 
Recalling some results in \cite{Cannarsa_Chen_Cheng2019}, assuming uniqueness for generalized characteristics, one can bridge the Aubry set (Mather set) and the invariant set of the associated semi-flow of generalized characteristics. Recently, relations between propagation of singularities and global dynamics of lower dimensional Hamiltonian systems have  also been pointed out in \cite{Zhang2020}. More concrete applications to problems from Hamiltonian dynamical systems in the scheme of Mather theory and weak KAM theory are expected, including applications to the study of Burgers turbulence as noted in \cite{Khanin_Sobolevski2016}.

\bibliographystyle{alpha}
\bibliography{mybib}

\end{document}